\pdfoutput=1
\documentclass[11pt]{article}
\usepackage{epic,latexsym}
\usepackage{color}
\usepackage{tikz}
\usepackage{comment}
\usepackage{amssymb}
\usepackage{mathtools}
\usepackage{amsthm}
\usepackage{url}

\newtheorem{theorem}{Theorem}[section]
\newtheorem{definition}[theorem]{Definition}
\newtheorem{proposition}[theorem]{Proposition}
\newtheorem{observation}[theorem]{Observation}

\newtheorem{corollary}[theorem]{Corollary}

\newtheorem{claim}{Claim}
\newtheorem{subclaim}{Claim}[claim]

\newcommand{\gs}{\gamma_{\rm S}}
\newcommand{\Gs}{\Gamma_{\rm S}}

\newcommand{\modo}{{\rm mod \,}}
\newcommand{\2}{ \vspace{0.2cm} }
\newcommand{\1}{ \vspace{0.1cm} }
\newcommand{\smallqed}{{\tiny ($\Box$)}}

\let\oldenumerate\enumerate
\renewcommand{\enumerate}{
  \oldenumerate
  \setlength{\itemsep}{0pt}
  \setlength{\parskip}{0pt}
  \setlength{\parsep}{0pt}
}

\begin{document}

\title{The Sierpi\'{n}ski Domination Number}

\author{$^1$Michael A. Henning, \,
$^{2,3,4}$Sandi Klav\v zar, \\ $^{1,5}$El\.{z}bieta Kleszcz,  \, and  \, $^5$Monika Pil\'{s}niak
\\ \\
\small $^1$Department of Mathematics and Applied Mathematics \\
\small University of Johannesburg \\
\small Auckland Park, 2006 South Africa\\
\small \tt e-mail: mahenning@uj.ac.za
\\ \\
\small $^2$Faculty of Mathematics and Physics \\
\small University of Ljubljana, Slovenia\\
\small \tt e-mail: sandi.klavzar@fmf.uni-lj.si
\\ \\
\small $^3$Institute of Mathematics, Physics and Mechanics \\
\small Ljubljana, Slovenia
\\ \\
\small $^4$ Faculty of Natural Sciences and Mathematics \\
\small University of Maribor, Slovenia 
\\ \\
\small $^5$AGH University\\
\small Department of Discrete Mathematics \\
\small al. Mickiewicza 30, 30-059 Krakow, Poland\\
\small \tt e-mail: {pilsniak, elzbieta.kleszcz}@agh.edu.pl
\\
}

\date{}
\maketitle

\begin{abstract}
Let $G$ and $H$ be graphs and let $f \colon V(G)\rightarrow V(H)$ be a function. The Sierpi\'{n}ski product of $G$ and $H$ with respect to $f$, denoted by  $G \otimes _f H$, is defined as the graph on the vertex set $V(G)\times V(H)$, consisting of $|V(G)|$ copies of $H$; for every edge $gg'$ of $G$ there is an edge between copies $gH$ and $g'H$ of $H$ associated with the vertices $g$ and $g'$ of $G$, respectively, of the form $(g,f(g'))(g',f(g))$. In this paper, we define the Sierpi\'{n}ski domination number as the minimum of $\gamma(G\otimes _f H)$ over all functions $f \colon V(G)\rightarrow V(H)$. The upper Sierpi\'{n}ski domination number is defined analogously as the corresponding maximum. After establishing general upper and lower bounds, we determine the upper Sierpi\'{n}ski domination number of the Sierpi\'{n}ski product of two cycles, and determine the lower Sierpi\'{n}ski domination number of the Sierpi\'{n}ski product of two cycles in half of the cases and in the other half cases restrict it to two values. 
\end{abstract}

{\small \textbf{Keywords:} Sierpi\'{n}ski graph; Sierpi\'{n}ski product; domination number; Sierpi\'{n}ski domination number} \\
\indent {\small \textbf{AMS subject classification:} 05C69, 05C76}

\section{Introduction}

Sierpi\'{n}ski graphs represent a very interesting and widely studied family of graphs. They were introduced in 1997 in the paper~\cite{klavzar-1997}, where the primary motivation for their introduction was the intrinsic link to the Tower of Hanoi problem, for the latter problem see the book~\cite{hinz-2018}. Intensive research of Sierpi\'{n}ski graphs led to a review article~\cite{hinz-2017} in which state of the art up to 2017 is summarized and unified approach to Sierpi\'{n}ski-type graph families is also proposed. Later research on Sierpi\'{n}ski graphs includes~\cite{bresar-2018, deng-2021, farrokhi-2021, liu-2021, varghese-2021}.

Sierpi\'{n}ski graphs have a fractal structure, the basic graphs of which are complete graphs. In 2011, Gravier, Kov\v se, and Parreau~\cite{gravier-2011} introduced a generalization in such a way that any graph can act as a fundamental graph, and called the resulting graphs generalized Sierpi\'{n}ski graphs. We refer to the papers~\cite{balakrishnan-2022, estrada-moreno-2018, estrada-moreno-2019, imrich-2020, khatibi-2020, klavzar-2018, korze-2019, menon-2023, ramezani-2017, rodriguez-2017, vatandoost-2019} for investigations of generalized Sierpi\'{n}ski graphs in the last few years.

An interesting generalization of Sierpi\'{n}ski graphs in the other direction has recently been proposed by Kovi\v{c}, Pisanski, Zemlji\v{c}, and \v{Z}itnik in~\cite{kpzz-2022}. Namely, in the spirit of classical graph products, where the vertex set of a product graph is the Cartesian product of the vertex sets of the factors, they introduced the Sierpi\'{n}ski product of graphs as follows. Let $G$ and $H$ be graphs and let $f \colon V(G)\rightarrow V(H)$ be an arbitrary function. The \textit{Sierpi\'{n}ski product of graphs $G$ and $H$ with respect to $f$}, denoted by  $G \otimes _f H$, is defined as the graph on the vertex set $V(G)\times V(H)$ with edges of two types:
\begin{itemize}
    \item \emph{type-$1$ edge}:
    $(g,h)(g,h')$ is an edge of  $G \otimes _f H$ for every vertex $g\in V(G)$ and every edge $hh' \in E(H)$,
    \item \emph{type-$2$ edge}: $(g,f(g'))(g',f(g))$ is an edge of $G \otimes _f H$ for every edge $gg' \in E(G)$.
\end{itemize}

We observe that the edges of type-$1$ induce $n(G) = |V(G)|$ copies of the graph $H$ in the Sierpi\'{n}ski product $G \otimes _f H$. For each vertex $g \in V(G)$, we let $gH$ be the copy of $H$ corresponding to the vertex $g$. A type-$2$ edge joins vertices from different copies of $H$ in $G \otimes _f H$, and is called a \emph{connecting edges} of $G \otimes _f H$. A vertex incident with a connecting edge is called a \emph{connecting vertex}. We observe that two different copies of $H$ in $G \otimes _f H$ are joined by at most one edge. A copy of the graph $H$ corresponding to a vertex of the graph $G$ in the Sierpi\'{n}ski product $G \otimes _f H$ is called an \textit{$H$-layer}.

Let $G$ and $H$ be graphs and $H^G$ be the family of functions from $V(G)$ to $V(H)$. We introduce new types of domination, the \textit{Sierpi\'{n}ski domination number}, denoted by $\gamma_{S}(G, H)$, as the minimum over all functions $f$ from $H^G$ of the domination number of the Sierpi\'{n}ski product with respect to $f$, and \textit{upper Sierpi\'{n}ski domination number}, denoted by $\Gamma_{S}(G, H)$, as the maximum over all functions $f\in H^G$ of domination number of the Sierpi\'{n}ski product with respect to $f$. That is,
\[
\gs(G, H) \coloneqq \min_{f\in H^G}\{\gamma(G\otimes _f H)\}
\]
and
\[
\Gs(G, H) \coloneqq \max _{f\in H^G}\{\gamma(G\otimes _f H)\}\,.
\]

In this paper, we initiate the study of Sierpi\'{n}ski domination in graphs. In Section~\ref{S:notation} we present the graph theory notation and terminology we follow. In Section~\ref{S:elementary} we discuss general lower and upper bounds on the (upper) Sierpi\'{n}ski domination number. Our main contribution in this introductory paper is to
determine the upper Sierpi\'{n}ski domination number of the Sierpi\'{n}ski product of two cycles, and to determine the lower Sierpi\'{n}ski domination number of the Sierpi\'{n}ski product of two cycles in half of the cases and in the other half cases restrict it to two values. 

\subsection{Notation and terminology}
\label{S:notation}

We generally follow the graph theory notation and terminology in the books~\cite{HaHeHe-20,HaHeHe-21,HeYe-book} on domination in graphs. Specifically, let $G$ be a graph with vertex set $V(G)$ and edge set $E(G)$, and of order $n(G) = |V(G)|$ and size $m(G) = |E(G)|$. For a subset $S$ of vertices of a graph $G$, we denote by $G - S$ the graph obtained from $G$ by deleting the vertices in $S$ and all edges incident with vertices in $S$. If $S = \{v\}$, then we simply write $G - v$ rather than $G - \{v\}$. The subgraph induced by the set $S$ is denoted by $G[S]$. We denote the path, cycle and complete graph on $n$ vertices by $P_n$, $C_n$, and $K_n$, respectively. For $k \ge 1$ an integer, we use the notation $[k] = \{1,\ldots,k\}$ and $[k]_0 = \{0,1,\ldots,k\}$. We generally label vertices of the considered graphs by elements of $[n]$. In this case, the mod function over the set $[n]$ is to be understood in a natural way, more formally, we apply the following operation for $t\geq 1$: $t\mathrm{\; mod^*\;}n = (t-1)\mathrm{\; mod\;} n+1$. 

A vertex \emph{dominates} itself and its neighbors, where two vertices are neighbors in a graph if they are adjacent. A \emph{dominating set} of a graph $G$ is a set $S$ of vertices of $G$ such that every vertex in $G$ is dominated by a vertex in $S$. The \emph{domination number}, $\gamma(G)$, of $G$ is the minimum cardinality of a dominating set of $G$. A dominating set of cardinality $\gamma(G)$ is called a $\gamma$-\emph{set of $G$}. A thorough treatise on dominating sets can be found in \cite{HaHeHe-20,HaHeHe-21}.

If $S$ is a set of vertices in a graph $G$, then we will use the notation $G|S$ to denote that the vertices in the set $S$ are assumed to be dominated and hence $\gamma(G|S)$ is the minimum number of vertices in the graph $G$ needed to dominate $V(G) \setminus S$. We note that it could be that a vertex in $S$ is still a member of a such a minimum dominating set no matter that we do not need to dominate the vertices in $S$ themselves. If $S = \{x\}$, then we simply denote $G|S$ by $G|x$ rather than $G|\{x\}$.

\section{General lower and upper bounds}
\label{S:elementary}

We present in this section general lower and upper bounds on the (upper) Sierpi\'{n}ski domination number.

\begin{theorem}
\label{thm:elementary}
If $G$ and $H$ are graphs, then
\[
n(G)\gamma(H)-m(G)\le \gs(G,H)\le \Gs(G,H) \le n(G)\gamma(H)\,.
\]
\end{theorem}
\begin{proof}
Let $G \otimes _f H$ be an arbitrary Sierpi\'{n}ski product of graphs $G$ and $H$ and let $X$ be a $\gamma$-set of $G \otimes _f H$. Assuming for a moment that all the connecting edges are removed from $G \otimes _f H$, we obtain $n(G)$ disjoint copies of $H$ for which we clearly need $n(G)\gamma(H)$ vertices in a minimum dominating set. Consider now an arbitrary connecting edge $e=(g,f(g'))(g',f(g))$ of $G \otimes _f H$. If no end-vertex of $e$ lies in $X$, then clearly $\gamma(G\otimes _f H- e) = \gamma(G\otimes _f H)$. Similarly, if both  end-vertices of $e$ lie in $X$, then $\gamma(G\otimes _f H- e) = \gamma(G\otimes _f H)$. Hence the only situation in which $e$ has an effect on $\gamma(G\otimes _f H)$ is when $(g,f(g'))\in X$ and $(g',f(g))\notin X$ (or the other way around). But in this case, the effect of the presence of the edge $e$ is that because $(g,f(g'))$ dominates one vertex of $g'H$, the edge $e$ might reduce the domination number by~$1$. That is, each connecting edge can drop the domination number of $G \otimes _f H$ by at most $1$, which proves the left inequality. The other two inequalities are clear.
\end{proof}

To show that the lower bound of Theorem~\ref{thm:elementary} is achieved, we show later in Theorem~\ref{t:thm-3k1dom} that for $n \ge 3$ and $k \ge 1$, if we take $G = C_n$ and $H = C_{3k+1}$ where $n \equiv 0 \, (\modo \, 4)$, then $\gs(G,H) = kn = n(G)\gamma(H) - m(G)$. The upper bound of Theorem~\ref{thm:elementary} is obtained, for example, for the Sierpi\'{n}ski product of two complete graphs. More generally, to achieve equality in the upper bound of Theorem~\ref{thm:elementary} we require the graph $H$ to have the following property.

\begin{theorem}
\label{thm:equ-upper}
The equality in $\Gs(G,H)\le n(G)\gamma(H)$ is achieved if and only if there exists a vertex $x\in V(H)$ such that $\gamma(H|x)=\gamma(H)$.
\end{theorem}
\begin{proof}
Suppose that $H$ has a vertex $x$ that satisfies $\gamma(H|x)=\gamma(H)$. In this case, we consider the Sierpi\'{n}ski product $G \otimes_f H$ with the function $f \colon V(G)\rightarrow V(H)$ defined by $f(v) = x$ for every vertex $v \in V(G)$. Consequently each connecting edge in the product is of the form $(g,x)(g',x)$. Thus, if $X$ is a $\gamma$-set of $G \otimes_f H$, then $|X \cap V(gH)| = \gamma(H)$ because the only vertex of $gH$ that can be dominated from outside $gH$ is $(g,x)$, but we have assume that $\gamma(H|x) = \gamma(H)$. Therefore, $\Gamma_S(G,H) = n(G)\gamma(H)$.

For the other implication suppose that $\gamma(H|x)<\gamma(H)$ for every vertex $x \in V(H)$. For an arbitrary edge $g_1g_2 \in E(G)$, if $D$ corresponds to a $\gamma$-set of the product $G \otimes_f H$, then $|D \cap V(G \otimes_f H[V(g_1H)\cup V(g_2H)])| \le 2\gamma(H)-1$. Consequently, $\Gamma_S(G,H) < n(G)\gamma(H)$.
\end{proof}

To conclude this section we describe large classes of graphs for which the second and the third inequality of Theorem~\ref{thm:elementary} are both equality. 

\begin{proposition}
If $G$ and $H$ are graphs such that $\Delta(G) < n(H)$ and $\gamma(H)=1$, then $\Gs(G,H) = \gs(G,H) = n(G)$.
\end{proposition}
\begin{proof}
Let $G$ and $H$ be graphs such that $\Delta(G) < n(H)$ and $\gamma(H)=1$. Thus, by Theorem \ref{thm:elementary}, $\gs(G,H) \leq \Gs(G,H)\leq n(G)$ and it is straightforward that the Sierpi\'{n}ski product of graphs $G$ and $H$ can be dominated by taking one dominating vertex from each $H$-layer to the dominating set. 
It remains to show that the inequality $\gs(G,H)\geq n(G)$ also holds. Suppose that $\gs(G,H)\leq n(G)-1$. Let $D$ be a dominating set of $G \otimes _f H$, where $f$ is such that that minimizes the domination number. Therefore there is an $H$-layer, denote it by $H'$, of $G \otimes _f H$ such that $D\cap V(H')=\emptyset$. Since there are at most $\Delta(G)$ connecting edges incident with vertices from each $H$-layer and $\Delta(G) < n(H)$,then all the vertices from an $H$-layer cannot be dominated by the vertices from the neighboring layers. Therefore we have $D\cap V(H'')\neq \emptyset$ for each $H''$-layer of $G \otimes _f H$. Thus $\gs(G,H) \geq n(G)$ and the result follows.
\end{proof}

\section{The Sierpi\'{n}ski domination number of cycles}

Let us recall firstly the domination number of a path and a cycle.

\begin{observation}
For $n\ge 3$, $\gamma(P_n)=\gamma(C_n)=\big\lceil \frac{n}{3}\big\rceil$.
\end{observation}

In this section, we shall prove the following results.

\begin{theorem}
\label{t:thm-main-1}
For $n \ge 3$, $k \ge 1$, and $p \in [2]_0$,
\[
\gs(C_n,C_{3k+p}) \in
\left\{
\begin{array}{ll}
\{kn\};  & \mbox{$p = 0$,} \2 \\
\{kn, kn+1\};  & \mbox{$p = 1$,} \2 \\
\{kn + \left\lfloor \dfrac{n}{2} \right\rfloor, kn + \left\lfloor \dfrac{n}{2} \right\rfloor+1\};  & \mbox{$p = 2$.}
\end{array}
\right.
\]
Moreover, if $n \equiv 0 \mod 4$, then $\gs(C_n,C_{3k+1})  = kn$ and $\gs(C_n,C_{3k+2})  = kn + \left\lfloor \dfrac{n}{2} \right\rfloor$.
\end{theorem}

\begin{theorem}
\label{t:thm-main-2}
For $n \ge 3$, $k \ge 1$, and $p \in [2]_0$,
\[
\Gs(C_n,C_{3k+p}) =
\left\{
\begin{array}{ll}
kn;  & \mbox{$p = 0$,} \2 \\
kn+ \left\lceil \dfrac{n}{3} \right\rceil\,;  & \mbox{$p = 1$,} \2 \\
(k+1)n;  & \mbox{$p = 2$.}
\end{array}
\right.
\]
\end{theorem}

In order to prove Theorems~\ref{t:thm-main-1} and~\ref{t:thm-main-2}, we consider three cases, depending on the value of~$p$.

\subsection{The cycle $C_n$ and cycles $C_{3k+1}$}
\label{S:3k+1}

To determine $\Gamma_{S}(C_n,C_{3k+1})$, we prove a slightly more general result. For this purpose, we define a class of graphs ${\cal H}_k$ as follows.

\begin{definition}
\label{defn1}
For $k \ge 1$, let ${\cal H}_k$ be the class of all graphs $H$ that have the following properties. \\ [-24pt]
\begin{enumerate}
\item $\gamma(H) = k+1$ and $\gamma(H-v) = k$ for every vertex $v \in V(H)$.
\item If $x,y \in V(H)$, then there exists a $\gamma$-set of $H$ that contains $x$ and~$y$, where $x=y$ is allowed.
\end{enumerate}
\end{definition}

We show, for example, that for every $k \ge 1$, the cycle $C_{3k+1}$ belongs to the class ${\cal H}_k$.

\begin{proposition}
\label{p:prop1}
For $k \ge 1$, the class ${\cal H}_k$ of graphs contains the cycle $C_{3k+1}$.
\end{proposition}
\begin{proof}
For $k \ge 1$, let $H \cong C_{3k+1}$. Since $\gamma(C_n) = \gamma(P_n) = \lceil n/3 \rceil$, property~(a) in Definition~\ref{defn1} holds. To prove that property~(b) in Definition~\ref{defn1} holds, let $x,y \in V(H)$. Since $H$ is vertex-transitive, every specified vertex belongs to some $\gamma$-set of $H$. In particular, if $x = y$, then property~(b) is immediate. Hence, we may assume that $x \ne y$. Let $H$ be the cycle $v_1 v_2 \ldots v_{3k+1} v_1$, where renaming vertices if necessary, we may assume that $x = v_1$. Let $y = v_i$, and so $i \in [3k+1] \setminus \{1\}$.

Let $H' = H - N[\{x,y\}]$, that is, $H'$ is obtained from $H$ by removing $x$ and $y$, and removing all neighbors of $x$ and $y$. If $H'$ is connected, then $H'$ is a path $P_{3(k-2) + j}$ for some $j$ where $j \in [3]$. In this case, $\gamma(H') = k-1$. If $H'$ is disconnected, then $H'$ is the disjoint union of two paths $P_{k_1}$ and $P_{k_2}$, where $k_1 + k_2 = 3(k-2)+1$. Thus renaming $k_1$ and $k_2$ if necessary, we may assume that either $k_1 = 3j_1$ and $k_2 = 3j_2 + 1$ where $j_1 \ge 1$, $j_2 \ge 0$, and $j_1 + j_2 = k-2$ or $k_1 = 3j_1+2$ and $k_2 = 3j_2 + 2$ where $j_1, j_2 \ge 0$ and $j_1 + j_2 = k-3$. In both cases, $\gamma(H') = \lceil k_1/3 \rceil + \lceil k_2/3 \rceil = k - 1$. Letting $D'$ be a $\gamma$-set of $H'$, the set $D = D' \cup \{x,y\}$ is a dominating set of $H$ of cardinality~$k+1 = \gamma(H)$, implying that $D$ is a $\gamma$-set of $H$ that contains both~$x$ and~$y$. Hence, property~(b) holds.
\end{proof}

For $n \ge 3$ an integer, a \emph{circulant graph} $C_n \langle L \rangle$ with a given list $L \subseteq \{1, \ldots, \lfloor \frac{1}{2}n \rfloor\}$ is a graph on $n$ vertices in which the $i$th vertex is adjacent to the $(i+j)$th and $(i-j)$th vertices for each $j$ in the list $L$ and where addition is taken modulo~$n$. For example, for $n = 3k+1$ where $k \ge 1$ and $L = \{1\}$, the circulant graph $C_n \langle L \rangle$ is the cycle $C_{3k+1}$, which, by Proposition~\ref{p:prop1}, belongs to the class ${\cal H}_k$. More generally, for $n = k(2p+1)+1$ where $k \ge 1$, $p \ge 1$, and $L = [p]$, the circulant graph $C_n \langle L \rangle$ belongs to the class ${\cal H}_k$. We omit the relatively straightforward proof. These examples of circulant graphs serve to illustrate that for each $k \ge 1$, one can construct infinitely many graphs in the class ${\cal H}_k$. We determine next the upper Sierpi\'{n}ski domination number $\Gamma_{S}(C_n,H)$ of a cycle $C_n$ and a graph $H$ in the family ${\cal H}_k$.

\begin{theorem}
\label{t:thm1}
For $n \ge 3$ and $k \ge 1$, if $H \in {\cal H}_k$, then
\[
\Gs(C_n,H) = kn+ \left\lceil \dfrac{n}{3} \right\rceil\,.
\]
\end{theorem}
\begin{proof}
For $n \ge 3$ and $k \ge 1$, let $G \cong C_n$ and let $H \in {\cal H}_k$. Let $G$ be the cycle given by $g_1 g_2 \ldots g_n g_1$. In what follows, we adopt the following notation. For each $i \in [n]$, we denote the copy $g_iH$ of $H$ corresponding to the vertex $g_i$ simply by $H_i$. We proceed further with two claims. The first claim establishes a lower bound on $\Gs(C_n,H)$, and the second claim establishes an upper bound on $\Gs(C_n,H)$.

\begin{claim}
\label{claim:1}
$\Gs(C_n,H) \ge kn + \left\lceil \dfrac{n}{3} \right\rceil\,$.
\end{claim}
\proof Let $f \colon V(G) \rightarrow V(H)$ be a constant function, that is, we select $h \in V(H)$ and for every vertex $g \in V(G)$, we set $f(g) = h$. Let $D_G$ be a $\gamma$-set of $G$. Thus, $|D_G| = \gamma(C_n) = \lceil n/3 \rceil$. By property~(a) in Definition~\ref{defn1}, for every vertex $g \in V(G)$, there exists a $\gamma$-set of $gH$ that contains the vertex~$(g,f(g)) = (g,h)$. If $g \in D_G$, let $D_g$ be a $\gamma$-set of $gH$ that contains the vertex~$(g,f(g)) = (g,h)$, and so $|D_g| = \gamma(H) = k+1$. If $g \in V(G) \setminus D_G$, let $D_g$ be a $\gamma$-set of $gH - (g,f(g)) = gH - (g,h)$, and so in this case $|D_g| = \gamma(H-h) = \gamma(H) - 1 = k$. Let
\[
D = \bigcup_{g \in V(G)} D_g.
\]

The set $D$ is a dominating set of $G \otimes _f H$, and so
\begin{equation}
\label{Eq1a}
\gamma(G \otimes _f H) \le |D| = \gamma(G) (k+1) + (n - \gamma(G)) k = kn + \gamma(G) = kn + \left\lceil \dfrac{n}{3} \right\rceil\,.
\end{equation}

For the fixed vertex $h$ chosen earlier, we note that the set of vertices $(g,h)$ for all $g \in V(G)$ induces a subgraph of $G \otimes _f H$ that is isomorphic to $G \cong C_n$. We denote this copy of $G$ by $Gh$. Among all $\gamma$-sets of $G \otimes _f H$, let $D^*$ be chosen to contain as many vertices of $Gh$ as possible. Let $D^*_g = D^* \cap V(gH)$ for every $g \in V(G)$. Further let $D^*_G = \{(g,h) \in D^* \colon g \in V(G)\}$, that is, $D^*_G$ is the restriction of $D^*$ to the copy of $G$. If a vertex $(g,h) \notin D^*_G$ and $(g,h)$ is not dominated by $D^*_G$, then $D^*_g$ is a $\gamma$-set of $gH$ by the minimality of the set $D^*$. However in this case, we could replace the set $D^*_g$ be a $\gamma$-set of $gH$ that contains the vertex $(g,h)$ to produce a new $\gamma$-set of $G \otimes _f H$ that contains more vertices from the copy of $G$ than does $D^*$, a contradiction. Hence, the set $D^*_G$ is a dominating set in the copy of $G$, and so $|D^*_g(G)| \ge \gamma(G)$. By the minimality of the set $D^*$ and by property~(a) in Definition~\ref{defn1}, for each vertex $g \in V(G)$, we have $|D^*_g| = \gamma(H) = k+1$ if the vertex $(g,h) \in D^*_G$ and $|D^*_g| = \gamma(H - h) = k$ if the vertex $(g,h) \notin D^*_G$. Therefore,
\begin{equation}
\label{Eq1b}
\gamma(G \otimes _f H) = |D^*| = |D^*_G| (k+1) + (n - |D^*_G|) k = kn + |D^*_G| \ge kn + \gamma(G) = kn + \left\lceil \dfrac{n}{3} \right\rceil\,.
\end{equation}

By inequalities~(\ref{Eq1a}) and~(\ref{Eq1b}), we have
\begin{equation}
\label{Eq1c}
\gamma(G \otimes _f H) = kn + \left\lceil \dfrac{n}{3} \right\rceil\,.
\end{equation}

By equation~(\ref{Eq1c}), we have $\Gs(C_n,H) \ge \gamma(G \otimes _f H)= kn + \lceil n/3 \rceil\,$. This completes the proof of Claim~\ref{claim:1}.~\smallqed \1

\begin{claim}
\label{claim:2}
$\Gs(C_n,H) \le kn + \left\lceil \dfrac{n}{3} \right\rceil\,$.
\end{claim}
\proof
Let $f \colon V(G) \rightarrow V(H)$ be an arbitrary function. Let $H_i$ be the $i$th copy of $H$ corresponding to the vertex $g_i$ of $G$ for all $i \in [n]$. Let $D$ be the dominating set of $G \otimes _f H$ constructed as follows. Let $x_iy_{i+1}$ be the connecting edge from $H_i$ to $H_{i+1}$ for all $i \in [n]$, where addition is taken modulo~$n$. Thus, the vertex $x_i \in V(H_i)$ is adjacent to the vertex $y_{i+1} \in V(H_{i+1})$ in the graph $G \otimes _f H$, that is, $x_i = (g_i,f(g_{i+1}))$ and $y_{i+1} = (g_{i+1},f(g_{i}))$. We note that possibly $x_i = y_i$. By property~(b) in Definition~\ref{defn1}, there exists a $\gamma$-set of $H_i$ that contains both $x_i$ and $y_i$. For $i \in [n]$, we define the sets $D_{i,1}$, $D_{i,2}$, and $D_{i,3}$ as follows. Let $D_{i,1}$ be a $\gamma$-set of $H_i - x_i$. Let $D_{i,2}$ be a $\gamma$-set of $H_i$ that contains both $x_i$ and $y_i$. Let $D_{i,3}$ be a $\gamma$-set of $H_i - y_i$. We note that $|D_{i,1}| = |D_{i,3}| = k$ and $|D_{i,2}| = k+1$. For $i \in [n]$, we define the set $D_i$ as follows.
\[
D_i =
\left\{
\begin{array}{cl}
D_{i,1};  & \mbox{$i \equiv 1 \, (\modo \, 3)$ and $i \ne n$,} \2 \\
D_{i,2};  & \mbox{$i \equiv 2 \, (\modo \, 3)$ or $i \equiv 1 \, (\modo \, 3)$ and $i = n$,} \2 \\
D_{i,3};  & \mbox{$i \equiv 0 \, (\modo \, 3)$}.
\end{array}
\right.
\]

For example, the set $D_1$ dominates all vertices of $H_1 - x_1$. The set $D_2$ contains the vertex $y_2$, which is adjacent to the vertex $x_1$ of $H_1$, and contains the vertex $x_2$, which is adjacent to the vertex $y_3$ of $H_3$, implying that $D_2$ dominates the vertex $x_1$ of $H_1$, all vertices of $H_2$, and the vertex $y_3$ of $H_3$. The set $D_3$ dominates all vertices of $H_3 - y_3$. Thus, $D_1 \cup D_2 \cup D_3$ dominates all vertices in $V(H_1) \cup V(H_2) \cup V(H_3)$ in the Sierpi\'{n}ski product $G \otimes _f H$. Moreover, $|D_1| + |D_2| + |D_3| = k + (k+1) + k = 3k+1$. More generally, the set $D_{3j-2} \cup D_{3j-1} \cup D_{3j}$ dominates all vertices in $V(H_{3j-2}) \cup V(H_{3j-1}) \cup V(H_{3j})$ in the Sierpi\'{n}ski product $G \otimes _f H$ for all $j \in \{1,\ldots,\lfloor n/3 \rfloor\}$. Moreover, $|D_{3j-2}| + |D_{3j-1}| + |D_{3j}| = k + (k+1) + k = 3k+1$. If $n \equiv 1 \, (\modo \, 3)$, then the set $D_n$ is a $\gamma$-set of $H_n$, and in this case $|D_n| = k+1$. If $n \equiv 2 \, (\modo \, 3)$, then the set $D_{n-1} \cup D_n$ dominates all vertices in $V(H_{n-1}) \cup V(H_n)$, and in this case $|D_{n-1}| + |D_n| = k + (k+1) = 2k+1$. The set
\[
D = \bigcup_{i=1}^n D_i
\]
is therefore a dominating set of $G \otimes _f H$, implying that
\[
\gamma(G \otimes _f H) \le |D| = \sum_{i=1}^n |D_i| = kn + \left\lceil \dfrac{n}{3} \right\rceil\,.
\]
This completes the proof of Claim~\ref{claim:2}.~\smallqed

\medskip
The proof of Theorem~\ref{t:thm1} follows as an immediate consequence of Claims~\ref{claim:1} and~\ref{claim:2}.
\end{proof}

\medskip
As a consequence of Proposition~\ref{p:prop1}, we have the following special case of Theorem~\ref{t:thm1}.

\begin{corollary}
\label{t:thm-3k1Udom}
For $n \ge 3$ and $k \ge 1$,
\[
\Gs(C_n,C_{3k+1}) = kn + \left\lceil \dfrac{n}{3} \right\rceil\,.
\]
\end{corollary}

We consider next the Sierpi\'{n}ski domination number of $C_n$ and $C_{3k+1}$, and show that $\gs(C_n, C_{3k+1}) = kn$ if $n \equiv 0 \, (\modo \, 4)$ and $\gs(C_n, C_{3k+1}) \in \{kn, kn + 1\}$, otherwise.

\begin{theorem}
\label{t:thm-3k1dom}
For $n \ge 3$ and $k \ge 1$,
\[
\gs(C_n,C_{3k+1}) \in \{ kn, kn+1\}.
\]
Moreover, if $n \equiv 0 \mod 4$, then $\gs(C_n,C_{3k+1})  = kn$.
\end{theorem}
\begin{proof}
For $n \ge 3$ and $k \ge 1$, let $G = C_n$ and let $H = C_{3k+1}$. Let $G$ be the cycle given by $g_1 g_2 \ldots g_n g_1$. We adopt our notation employed in our earlier proofs. For notational convenience, we let $V(H) = \{1,2,\ldots, 3k+1\}$ where vertices $i$ and $i+1$ are consecutive on the cycle $H$ for all $i \in [3k+1]$ (and where addition is taken modulo~$3k+1$, and so vertex $1$ and vertex $3k+1$ are adjacent).

As before, we denote the copy $g_iH$ of $H$ corresponding to the vertex $g_i$ simply by $H_i$ for each $i \in [n]$. Thus, $H_i = C_{3k+1}$ is the cycle $(g_i,1), (g_i,2), \ldots, (g_i,3k+1), (g_i,1)$ for all $i \in [n]$. Recall that we denote the connecting edge from $H_i$ to $H_{i+1}$ by $x_iy_{i+1}$ for all $i \in [n]$, where $x_i \in V(H_i)$, $y_{i+1} \in V(H_{i+1})$, and addition is taken modulo~$n$. Thus, $y_i = (g_i,f(g_{i-1}))$ and $x_i = (g_i,f(g_{i+1}))$ for all $i \in [n]$.

By Proposition~\ref{p:prop1}, the graph $H$ belongs to the class ${\cal H}_k$. Thus, $\gamma(H) = k+1$ and $\gamma(H-v) = k$ for every vertex $v \in V(H)$. Furthermore, if $x,y \in V(H)$ where $x=y$ is allowed, then there exists a $\gamma$-set of $H$ that contains $x$ and~$y$.


By the elementary lower bound on the Sierpi\'{n}ski domination number given in Theorem~\ref{thm:elementary}, $\gs(G,H) \ge n(G)\gamma(H)- m(G) = kn$, noting that here $n(G) = m(G) = n$ and $\gamma(H) = k+1$. It follows that $\gs(C_n,H) \ge kn$.

To complete the proof we are going to prove that 
$$\gs(C_n,H) \le kn + \left\lceil \dfrac{n}{4} \right\rceil - \left\lfloor \dfrac{n}{4} \right\rfloor\,.$$
Let $f \colon V(G) \rightarrow V(H)$ be the function defined by
\[
f(g_i) =
\left\{
\begin{array}{ll}
1; & \mbox{$i \bmod 4 \in \{1,2\}$}, \1 \\
3; & \mbox{otherwise}.
\end{array}
\right.
\]
for all $i \in [n]$ where addition is taken modulo~$n$. Adopting our earlier notation, recall that $y_i = (g_i,f(g_{i-1}))$ and $x_i = (g_i,f(g_{i+1}))$ for all $i \in [n]$. Let $n = 4\ell + j$ where $j \in [3]_0 = \{0,1,2,3\}$. We note that $f(g_{4i-3}) = f(g_{4i-2}) = 1$ and $f(g_{4i-1}) = f(g_{4i}) = 3$ for all $i \in [\ell]$. Let $D_i$ be the unique $\gamma$-set of $H_i - y_i \cong P_{3k}$ which consists of all vertices at distance~$2$ modulo~$3$ from $y_i$ in the graph $H_i$ for all $i \in [n]$, and let
\[
D = \bigcup_{i=1}^n D_i.
\]
We note that $|D_i| = k$ for all $i \in [n]$, and so $|D| = kn$. For all $i \in \{2,3,\ldots,\ell-1\}$, the following four properties hold. \\ [-24pt]
\begin{enumerate}
\item[] {\rm (P1)}: $y_{4i-3} = (g_{4i-3},3)$ and $x_{4i-3} = (g_{4i-3},1)$.
\item[] {\rm (P2)}: $y_{4i-2} = (g_{4i-2},1)$ and $x_{4i-2} = (g_{4i-2},3)$.
\item[] {\rm (P3)}: $y_{4i-1} = (g_{4i-1},1)$ and $x_{4i-1} = (g_{4i-1},3)$.
\item[] {\rm (P4)}: $y_{4i} = (g_{4i},3)$ and $x_{4i} = (g_{4i},1)$.
\end{enumerate}

Hence for all $i \in \{2,3,\ldots,\ell-1\}$, the vertices $x_i$ and $y_i$ are at distance~$2$ in $H_i$, implying that $x_i \in D_i$. We consider four cases to determine which properties hold for the boundary conditions (that is for $i\in\{1,\ell\})$ and finally to set the upper bound on the domination number in each case.

\medskip
\emph{Case~1. $n \equiv 0 \, (\modo \, 4)$, that is $n=4\ell$.}

\noindent In this case, properties P1 and P4 also hold for $i = 1$ and $i = \ell$, respectively. Thus, $y_{1} = (g_{1},3)$ and $x_{1} = (g_{1},1)$, and $y_{4\ell} = (g_{4\ell},3)$ and $x_{4\ell} = (g_{4\ell},1)$, implying that $x_1,x_{4\ell} \in D$. The set $D$ is therefore a dominating set of $G \otimes _f H$, and so $\gamma(G \otimes _f H) \le |D| = kn = kn + \lceil n/4 \rceil - \lfloor n/4 \rfloor$.

\medskip
\emph{Case~2. $n \equiv 1 \, (\modo \, 4)$, that is $n=4\ell+1$.}

\noindent In this case, $y_{1} = x_{1} = (g_{1},1)$, and $y_{4\ell+1} = (g_{4\ell+1},3)$ and $x_{4\ell+1} = (g_{4\ell+1},1)$. In particular, property P4 also holds for $i = \ell$, and so $x_{4\ell+1} \in D$. The set $D \cup \{x_1\}$ is therefore a dominating set of $G \otimes _f H$, and so $\gamma(G \otimes _f H) \le |D| + 1 = kn + 1 = kn + \lceil n/4 \rceil - \lfloor n/4 \rfloor$.

\medskip
\emph{Case~3. $n \equiv 2 \, (\modo \, 4)$, that is $n=4\ell+2$.}

\noindent In this case, $y_{1} = x_{1} = (g_{1},1)$, and $y_{4\ell+2} = x_{4\ell+2} = (g_{4\ell+2},1)$. We note that neither $x_1$ nor $x_{4\ell+2}$ belong to the set $D$. The set $D \cup \{x_1\}$ is a dominating set of $G \otimes _f H$, and so $\gamma(G \otimes _f H) \le |D| + 1 = kn + 1 = kn + \lceil n/4 \rceil - \lfloor n/4 \rfloor$.

\medskip
\emph{Case~4. $n \equiv 3 \, (\modo \, 4)$, that is $n=4\ell+3$.}

\noindent In this case, $y_{1} = (g_{1},3)$ and $x_{1} = (g_{1},1)$, and $y_{4\ell+3} = x_{4\ell+3} = (g_{4\ell+3},1)$. In particular, property P1 also holds for $i = 1$, and so $x_1 \in D$. However, $x_{4\ell+3} \notin D$. The set $D \cup \{x_{4\ell+3}\}$ is therefore a dominating set of $G \otimes _f H$, and so $\gamma(G \otimes _f H) \le |D| + 1 = kn + 1 = kn + \lceil n/4 \rceil - \lfloor n/4 \rfloor$.

In all four cases, $\gamma(G \otimes _f H) \le kn + \lceil n/4 \rceil - \lfloor n/4 \rfloor$. 
\end{proof}

\subsection{The cycle $C_n$ and cycles $C_{3k+2}$}

In this section, we determine the Sierpi\'{n}ski domination number $\gamma_{S}(C_n,C_{3k+2})$ and the upper Sierpi\'{n}ski domination number $\Gamma_{S}(C_n,C_{3k+2})$.

\begin{theorem}
\label{t:thm-3k2Udom}
For $n \ge 3$ and $k \ge 1$, we have $\Gs(C_n,C_{3k+2}) = (k+1)n$.
\end{theorem}
\begin{proof} For $n \ge 3$ and $k \ge 1$, let $G \cong C_n$ and let $H \cong C_{3k+2}$. Let $f \colon V(G) \rightarrow V(H)$ be a constant function, that is, we select $h \in V(H)$ and for every vertex $g \in V(G)$, we set $f(g) = h$. For each vertex $g \in V(G)$, let $H_g$ denote the copy of $H$ associated with the vertex~$g$. Let $D$ be a dominating set of $G \otimes _f H$, and let $D_g = D \cap V(H_g)$, and so $D_g$ is the restriction of $D$ to the copy $H_g$ of $H$. If the vertex $(g,h)$ does not belong to $D_g$, then $D_g$ dominates all vertices on the path $H_g - (g,h) \cong P_{3k+1}$, and so $|D_g| \ge \gamma(P_{3k+1}) = k + 1$. If the vertex $(g,h)$ does belong to $D_g$, then $D_g$ dominates all vertices on the cycle $H_g \cong C_{3k+2}$, and so $|D_g| \ge \gamma(C_{3k+2}) = k + 1$. In both cases, $|D_g| \ge k+1$. Therefore,
\[
\gamma(G \otimes _f H) = |D| = \sum_{g \in V(G)} |D_g| \ge (k+1)n,
\]
implying that $\Gs(C_n,C_{3k+2}) \ge (k+1)n$. By the upper bound in Theorem~\ref{thm:elementary}, we have $\Gs(G,H) \le n(G)\gamma(H) = (k+1)n$, noting that in this case $\gamma(H) = \gamma(C_{3k+2}) = k+1$. Consequently, $\Gs(C_n,C_{3k+2}) = (k+1)n$.
\end{proof}

\begin{theorem}
\label{t:thm-3k2dom}
For $n \ge 3$ and $k \ge 1$,
\[
\gs(C_n,C_{3k+2}) \in \{ kn + \left\lfloor \dfrac{n}{2} \right\rfloor, kn + \left\lfloor \dfrac{n}{2} \right\rfloor+1\}.
\]
Moreover, if $n \equiv 0 \mod 4$, then $\gs(C_n,C_{3k+2})  = kn + \left\lfloor \dfrac{n}{2} \right\rfloor$.

\end{theorem}
\begin{proof}
For $n \ge 3$ and $k \ge 1$, let $G \cong C_n$ and let $H \cong C_{3k+2}$. We adopt our notation employed in our earlier proofs. Thus, the cycle $G$ is given by $g_1 g_2 \ldots g_n g_1$, and  $V(H) = \{1,2,\ldots, 3k+2\}$ where vertices $i$ and $i+1$ are consecutive on the cycle $H$ for all $i \in [3k+2]$ (and where addition is taken modulo~$3k+2$, and so vertex $1$ and vertex $3k+2$ are adjacent). As before, we denote the copy $g_iH$ of $H$ corresponding to the vertex $g_i$ simply by $H_i$ for each $i \in [n]$. Thus, $H_i = C_{3k+2}$ is the cycle $(g_i,1), (g_i,2), \ldots, (g_i,3k+2), (g_i,1)$ for all $i \in [n]$.

We adopt our notation from the proof of Theorem~\ref{t:thm1}. Thus, we denote the connecting edge from $H_i$ to $H_{i+1}$ by $x_iy_{i+1}$ for all $i \in [n]$, where $x_i \in V(H_i)$, $y_{i+1} \in V(H_{i+1})$, and addition is taken modulo~$n$. Thus, $y_i = (g_i,f(g_{i-1}))$ and $x_i = (g_i,f(g_{i+1}))$ for all $i \in [n]$.

We proceed further with two claims. The first claim establishes a lower bound on $\gs(G,H)$, and the second claim an upper bound on $\gs(G,H)$. Combining these two bounds yields the desired result in the statement of the theorem. \1

\begin{claim}
\label{claim-3k2dom-1}
$\gs(C_n,H) \ge kn + \left\lfloor \dfrac{n}{2} \right\rfloor$.
\end{claim}
\proof
Let $f \colon V(G) \rightarrow V(H)$ be an arbitrary function. We show that
\begin{equation}
\label{Eq3a}
\gamma(G \otimes _f H) \ge kn + \left\lfloor \dfrac{n}{2} \right\rfloor.
\end{equation}

Let $D$ be a $\gamma$-set of $G \otimes _f H$ constructed, and let $D_i = D \cap V(H_i)$ for $i \in [n]$. If the vertex $x_i$ is not dominated by $D_i$, then either $x_i \ne y_i$, in which case $x_i$ is dominated by the vertex $y_{i+1} \in D$, or $x_i = y_i$, in which case $x_i$ is dominated by the vertex $x_{i-1} \in D$ or the vertex $y_{i+1} \in D$. Analogously, if the vertex $y_i$ is not dominated by $D_i$, then either $x_i \ne y_i$, in which case $y_i$ is dominated by the vertex $x_{i-1} \in D$, or $x_i = y_i$, in which case $y_i$ is dominated by the vertex $x_{i-1} \in D$ or the vertex $y_{i+1} \in D$. If a vertex is not dominated by $D_i$, then such a vertex is $x_i$ or $y_i$, and we say that such a vertex is dominated from outside~$H_i$.

Similarly as before, we proceed with a claim that delivers properties of sets $D_i$ leading to the desired lower bound on the Sierpi\'{n}ski domination number.

\begin{subclaim}
\label{claim-3k2dom-1.1}
The following properties hold in the graph $H_i$. \\ [-22pt]
\begin{enumerate}
\item If $d(x_i,y_i) \equiv 1 \, (\modo \, 3)$, then $|D_i| = k$. Further, both $x_i$ and $y_i$ are dominated from outside $H_i$.
\item If $d(x_i,y_i) \not\equiv 1 \, (\modo \, 3)$, then $|D_i| = k+1$.
\end{enumerate}
\end{subclaim}
\proof
Suppose that $D_i$ contains a vertex $w_i$ that dominates $x_i$. Possibly, $w_i = x_i$. In order to dominate the $3(k-1) + 2$ vertices in $H_i$ not dominated by $w_i$, at least $k$ additional vertices are needed even if the vertex $y_{i}$ is dominated outside the cycle $H_i$. Thus in this case, $|D_i| \ge k+1$, implying by the minimality of the set $D$ that $|D_i| = k+1$. Analogously, if $D_i$ contains a vertex that dominates $y_i$, then $|D_i| = k+1$. Hence, if $x_i$ or $y_i$ (or both $x_i$ and $y_i$) are dominated by $D_i$, then $|D_i| = k+1$.

Suppose that neither $x_i$ nor $y_i$ is dominated by $D_i$, implying that both $x_i$ and $y_i$ are dominated from outside the cycle $H_i$. Thus, $D_i$ is a dominating set of $H_i' = H_i - x_i - y_i$. If $x_i = y_i$, then $H_i' = P_{3k+1}$, and by the minimality of $D$ we have $|D_i| = \gamma(P_{3k+1}) = k+1$. Hence, we may assume that $x_i \ne y_i$. If $x_i$ and $y_i$ are adjacent, then $H_i' = P_{3k}$, and by the minimality of $D$ we have $|D_i| = \gamma(P_{3k}) = k$. Suppose that $x_i$ and $y_i$ are not adjacent, and so $H'$ is the disjoint union of two paths $P_{k_1}$ and $P_{k_2}$, where $k_1 + k_2 = 3k$.  If $k_1 = 3j_1 + 1$ and $k_2 = 3j_2 + 2$ (or if $k_1 = 3j_1 + 2$ and $k_2 = 3j_2 + 1$) for some integers $j_i$ and $j_2$ where $j_1 + j_2 = k-1$, then $|D_i| = \lceil k_1/3 \rceil + \lceil k_2/3 \rceil = (j_1 + 1) + (j_2 + 1) = k+1$. If $k_1 = 3j_1$ and $k_2 = 3j_2$ where $j_1 + j_2 = k$, then $|D_i| = \lceil k_1/3 \rceil + \lceil k_2/3 \rceil = j_1 + j_2 = k$. Hence if neither $x_i$ nor $y_i$ is dominated by $D_i$, then either $d(x_i,y_i) \equiv 1 \, (\modo \, 3)$, in which case $|D_i| = k$, or $d(x_i,y_i) \not\equiv 1 \, (\modo \, 3)$, in which case $|D_i| = k+1$. This proves properties~(a) and~(b) of the claim.~\smallqed

\medskip
By Claim~\ref{claim-3k2dom-1.1}, if $|D_i| = k$ for some $i \in [n]$, then $|D_{i-1}| = |D_{i+1}| = k+1$ where addition is taken modulo~$n$. Furthermore in this case when $|D_i| = k$, the vertices $x_i$ and $y_i$ are distinct and are both dominated from outside $H_i$, implying that $y_{i+1} \in D_{i+1}$ and $x_{i-1} \in D_{i-1}$. This implies that if $n$ is even, then $|D| \ge kn + n/2$, and if $n$ is odd, then $|D| \ge kn + (n+1)/2$. This proves  inequality~(\ref{Eq3a}).

\begin{claim}
\label{claim-3k2dom-2}
$\gs(C_n,H) \le kn + \left\lfloor \dfrac{n}{2} \right\rfloor +  \left\lceil \dfrac{n}{4} \right\rceil - \left\lfloor \dfrac{n}{4} \right\rfloor$.
\end{claim}
\proof
Let $f \colon V(G) \rightarrow V(H)$ be the function defined by
\[
f(g_i) =
\left\{
\begin{array}{ll}
1; & \mbox{$i \equiv 1 \, (\modo \, 4)$}, \1 \\
2; & \mbox{$i \equiv 2 \, (\modo \, 4)$}, \1 \\
3; & \mbox{otherwise}.
\end{array}
\right.
\]
for all $i \in [n]$ where addition is taken modulo~$n$. Adopting our earlier notation, recall that $y_i = (g_i,f(g_{i-1}))$ and $x_i = (g_i,f(g_{i+1}))$ for all $i \in [n]$. Let $n = 4\ell + j$ where $j \in [3]_0 = \{0,1,2,3\}$. We note that $f(g_{4i-3}) = 1$, $f(g_{4i-2}) = 2$, and $f(g_{4i-1}) = f(g_{4i}) = 3$ for all $i \in [\ell]$.

\medskip
\emph{Case~1. $n \equiv 0 \, (\modo \, 4)$.}

\noindent Thus, $n = 4\ell$. We note that $y_{4i-3} = (g_{4i-3},3)$ and $x_{4i-3} = (g_{4i-3},2)$ for all $i \in [n]$, and so in the graph $H_{4i-3}$ the vertices $x_{4i-3}$ and $y_{4i-3}$ are at distance~$1$. Moreover, $y_{4i-1} = (g_{4i-1},2)$ and $x_{4i-1} = (g_{4i-1},3)$ for all $i \in [n]$, and so in the graph $H_{4i-1}$ the vertices $x_{4i-1}$ and $y_{4i-1}$ are at distance~$1$. This implies that $H_{4i-j} - \{x_{4i-j},y_{4i-j}\} \cong C_{3k}$ for $j \in \{1,3\}$. Let $D_{4i-j}$ be a $\gamma$-set of $H_{4i-j} - \{x_{4i-j},y_{4i-j}\}$ for $j \in \{1,3\}$, and so $|D_{4i-j}| = k$.

We also note that $y_{4i-2} = (g_{4i-2},1)$ and $x_{4i-2} = (g_{4i-2},3)$ for all $i \in [n]$, and so in the graph $H_{4i-2}$ the vertices $x_{4i-2}$ and $y_{4i-2}$ are at distance~$2$. Moreover, $y_{4i} = (g_{4i},3)$ and $x_{4i} = (g_{4i},1)$ for all $i \in [n]$, and so in the graph $H_{4i}$ the vertices $x_{4i}$ and $y_{4i}$ are at distance~$2$. This implies that $H_{4i-j} - N[\{x_{4i-j},y_{4i-j}\}] \cong C_{3(k-1)}$ for $j \in \{0,2\}$. Let $D_{4i-j}$ be a $\gamma$-set of $H_{4i-j}$ that contains both vertices $x_{4i-j}$ and $y_{4i-j}$ for $j \in \{0,2\}$, and so $|D_{4i-j}| = k+1$. The set
\[
D = \bigcup_{i=1}^{4\ell} D_i
\]
is a dominating set of $G \otimes _f H$, and so $\gamma(G \otimes _f H) \le |D| = 4k \ell + 2\ell = kn + n/2$.

\medskip
\emph{Case~2. $n \equiv 2 \, (\modo \, 4)$.}

\noindent Thus, $n = 4\ell + 2$ and in this case, $f(g_{4\ell + 1}) = 1$ and $f(g_{4\ell + 2}) = 2$. We note that in the graph $H_{4\ell + 1}$, the vertices $x_{4\ell + 1}$ and $y_{4\ell + 1}$ are at distance~$1$ and in the graph $H_{4\ell + 2}$ we have $x_{4\ell + 2}=y_{4\ell + 2}$. For $i \in [4\ell]$, we define the set $D_i$ exactly as in the previous case. Further, let $D_{4\ell + 1}$ be a $\gamma$-set of  $H_{4\ell + 1} - \{x_{4\ell + 1},y_{4\ell + 1}\} \cong C_{3k}$, and let $D_{4\ell + 2}$ be a $\gamma$-set of  $H_{4\ell + 2}$ containing $x_{4\ell + 2}$. We note that $|D_{4\ell + 1}| = k$ and $|D_{4\ell + 2}| = k+1$. The set
\[
D = \bigcup_{i=1}^{4\ell+2} D_i
\]
is a dominating set of $G \otimes _f H$, and so $\gamma(G \otimes _f H) \le |D| = 4k \ell +2k + 2\ell + 1 = kn + n/2$.

\medskip
\emph{Case~3. $n \equiv 1 \, (\modo \, 4)$.}

\noindent Thus, $n = 4\ell + 1$, and in this case, $f(g_{4\ell + 1}) = 1$. Thus, $y_{4\ell + 1} = (g_{4\ell + 1},3)$ and $x_{4\ell + 1} = (g_{4\ell+1},1)$, and so in the graph $H_{4\ell + 1}$, the vertices $x_{4\ell + 1}$ and $y_{4\ell + 1}$ are at distance~$2$. For $i \in [4\ell]$, we define the set $D_i$ exactly as in the previous cases. Further, let $D_{4\ell + 1}$ be a $\gamma$-set of $H_{4\ell + 1}$ that contains the vertex $x_{4\ell + 1}$. We note that $|D_{4\ell + 1}| = k+1$. The set
\[
D = \bigcup_{i=1}^{4\ell+1} D_i
\]
is a dominating set of $G \otimes _f H$, and so $\gamma(G \otimes _f H) \le |D| = 4k\ell + k + 2\ell + 1 = kn + (n+1)/2$.

\medskip
\emph{Case~4. $n \equiv 3 \, (\modo \, 4)$.}

\noindent Thus, $n = 4\ell + 3$, and in this case, $f(g_{4\ell + 1}) = 1$, $f(g_{4\ell + 2}) = 2$, and $f(g_{4\ell + 3}) = 3$. In particular, $y_{4\ell + 3} = (g_{4\ell + 3},2)$ and $x_{4\ell + 3} = (g_{4\ell + 3},1)$, and so in the graph $H_{4\ell + 3}$, the vertices $x_{4\ell + 3}$ and $y_{4\ell + 3}$ are at distance~$1$. For $i \in [4\ell + 2]$, we define the set $D_i$ exactly as in Case~2. Further, let $D_{4\ell + 3}$ be a $\gamma$-set of $H_{4\ell + 3}$ containing the vertex $x_{4\ell + 3}$. We note that $|D_{4\ell + 3}| = k+1$. The set
\[
D = \bigcup_{i=1}^{4\ell+3} D_i
\]
is a dominating set of $G \otimes _f H$, and so $\gamma(G \otimes _f H) \le |D| = 4k\ell + 3k + 2\ell + 2 = kn + (n+1)/2$. The desired result of the claim now follows from the four cases above.~\smallqed

\medskip
The proof of Theorem~\ref{t:thm-3k2dom} follows as an immediate consequence of Claim~\ref{claim-3k2dom-1} and Claim~\ref{claim-3k2dom-2}.
\end{proof}

\subsection{The cycle $C_n$ and cycles $C_{3k}$}

In this section, we determine the Sierpi\'{n}ski domination number $\gamma_{S}(C_n,C_{3k})$ and the upper Sierpi\'{n}ski domination number $\Gamma_{S}(C_n,C_{3k})$.

\begin{theorem}
\label{t:thm-3kdom}
For $n \ge 3$ and $k \ge 1$,
\[
\gs(C_n,C_{3k}) = \Gs(C_n,C_{3k}) = kn.
\]
\end{theorem}
\begin{proof}
We adopt our notation from the earlier sections. Let $G\cong C_n$ be the cycle $g_1 g_2 \ldots g_n g_1$, and let $H_i$ be the $i$th copy of $C_{3k}$ corresponding to the vertices $g_i$ of $G$ for $i \in [n]$. As before, we denote the connecting edge from $H_i$ to $H_{i+1}$ by $x_iy_{i+1}$ for all $i \in [n]$.

Let $f \colon V(G) \rightarrow V(H)$ be an arbitrary function. Let $D$ be a $\gamma$-set of $G \otimes _f H$, and let $D_i = D \cap V(H_i)$ for $i \in [n]$. We show that $|D_i| = k$ for all $i \in [n]$. If both vertices $x_i$ and $y_i$ are dominated by $D_i$, then $D_i$ is a $\gamma$-set of $H_i \cong C_{3k}$, and so $|D_i| = k$. If exactly one of $x_i$ and $y_i$ is dominated by $D_i$, say $x_i$, then by the minimality of the set $D$, the set $D_i$ is a $\gamma$-set of $H_i - y_i \cong P_{3k-1}$, and so $|D_i| = k$. Hence, we may assume that neither $x_i$ nor $y_i$ is dominated by $D_i$, for otherwise, $|D_i| = k$ and the desired bound follows.

With our assumption that neither $x_i$ nor $y_i$ is dominated by $D_i$, the set $D_i$ is a $\gamma$-set of $H_i' = H_i - \{x_i,y_i\}$. If $x_i = y_i$, then $H_i' = P_{3k-1}$, and by the minimality of $D$ we have $|D_i| = \gamma(P_{3k-1}) = k$. Hence, we may assume that $x_i \ne y_i$. If $x_i$ and $y_i$ are adjacent, then $H_i' = P_{3k-2}$, and by the minimality of $D$ we have $|D_i| = \gamma(P_{3k-2}) = k$. Suppose that $x_i$ and $y_i$ are not adjacent, and so $H'$ is the disjoint union of two paths $P_{k_1}$ and $P_{k_2}$, where $k_1 + k_2 = 3k-2$.  If $k_1 = 3j_1 + 1$ and $k_2 = 3j_2$ for some integers $j_i$ and $j_2$ where $j_1 + j_2 = k-1$, then $|D_i| = \lceil k_1/3 \rceil + \lceil k_2/3 \rceil = (j_1 + 1) + j_2 = k$. Analogously, if $k_1 = 3j_1$ and $k_2 = 3j_2 + 1$, then $|D_i| = k$. If $k_1 = 3j_1 + 2$ and $k_2 = 3j_2 + 2$ for some integers $j_i$ and $j_2$ where $j_1 + j_2 = k-2$, then $|D_i| = \lceil k_1/3 \rceil + \lceil k_2/3 \rceil = (j_1 + 1) + (j_2 + 1) = k$. In all cases, $|D_i| = k$, implying that
\[
\gamma(G \otimes _f H) = |D| = \sum_{i=1}^n |D_i| = kn.
\]

Since $f \colon V(G) \rightarrow V(H)$ was chosen as an arbitrary function, and $D$ as an arbitrary $\gamma$-set of $G \otimes _f H$, we deduce that $\gs(C_n,C_{3k}) = \Gs(C_n,C_{3k}) = \gamma(G \otimes _f H) = kn$.
\end{proof}

\section*{Concluding remarks}

It seems to us that in the vast majority of cases where the lower Sierpi\'{n}ski domination number of the Sierpi\'{n}ski product of two cycles is specified to two values exactly, the larger of the two is the correct value. However, the following example, which surprised us, demonstrates that there are also cases where the exact value is the smaller of the two possible values. 

Let $G \cong C_{18}$ with $V(G)=[18]$ and let $H \cong C_7$ with $V(H)=[7]$ and let the function $f: V(G)\rightarrow V(H)$ be defined as follows:
\begin{align*}
f(1) & = f(4) =f(5) =f(18) = 4, \\
f(2) & = f(3) =f(6) =f(7) = 2, \\
f(8) & = f(9) = 7, \\ 
f(10) & = f(11) = 5, \\
f(12) & = f(13) = 3, \\
f(14) & = f(15) = 1, \\
f(16) & = f(17) = 6.
\end{align*}

Then Theorem~\ref{t:thm-main-1} asserts that $\gamma (G \otimes_f H) \in \{36, 37\}$ and it is straightforward to check that the exact value is $\gamma (G \otimes_f H) = 36$. 

\section*{Acknowledgements}

We would like to thank the reviewer very much for careful reading of the paper and in particular for pointing to a subtle technical detail that led to the reformulation of one of the main theorems.

This research was obtained during the sabbatical visit of the first author at the University of Ljubljana and he thanks them for their kindness in hosting him. The first author also acknowledges that his research visit was supported in part by the University of Johannesburg and the South African National Research Foundation. S.K.\ acknowledges the financial support from the Slovenian Research Agency (research core funding No.\ P1-0297 and projects J1-2452, N1-0285).

\end{document}